\documentclass[reqno]{amsart}
\usepackage{pictexwd,dcpic}
\usepackage{hvfloat}
\usepackage{amssymb}

\newtheorem{theorem}{Theorem}[section]
\newtheorem{proposition}[theorem]{Proposition}
\newtheorem{lemma}[theorem]{Lemma}
\newtheorem{corollary}[theorem]{Corollary}

\theoremstyle{definition}

\theoremstyle{remark}
\newtheorem{remark}[theorem]{Remark}

\numberwithin{equation}{section}

\begin{document}

\title{Representations by $x_1^2+2x_2^2+x_3^2+x_4^2+x_1x_3+x_1x_4+x_2x_4$}

\author{Ick Sun Eum}
\address{Department of Mathematical Sciences, KAIST}
\curraddr{Daejeon 373-1, Korea} \email{zandc@kaist.ac.kr}
\thanks{}

\author{Dong Hwa Shin}
\address{Department of Mathematical Sciences, KAIST}
\curraddr{Daejeon 373-1, Korea} \email{shakur01@kaist.ac.kr}
\thanks{}

\author{Dong Sung Yoon}
\address{Department of Mathematical Sciences, KAIST}
\curraddr{Daejeon 373-1, Korea} \email{yds1850@kaist.ac.kr}
\thanks{}


\keywords{Eisenstein series, Hecke operators, modular forms, representations by quadratic forms.
\newline This research was partially supported by Basic Science Research
Program through the NRF of Korea funded by MEST (2010-0001654). The
second named author is partially supported by TJ Park Postdoctoral
Fellowship.}

\begin{abstract}
Let $r_Q(n)$ be the representation number of a nonnegative integer
$n$ by the quaternary quadratic form
$Q=x_1^2+2x_2^2+x_3^2+x_4^2+x_1x_3+x_1x_4+x_2x_4$. We first prove
the identity $r_Q(p^2n)=r_Q(p^2)r_Q(n)/r_Q(1)$ for any prime $p$
different from $13$ and any positive integer $n$ prime to $p$, which
was conjectured in \cite{E-K-S}. And, we explicitly determine a
concise formula for the number $r_Q(n^2)$ as well for any integer
$n$.
\end{abstract}

\maketitle

\section{Introduction}

Let $r$ be a positive integer and
\begin{equation*}
Q(x_1,\cdots,x_r)=\sum_{i<j}a_{ij}x_ix_j+\frac{1}{2}\sum_{i}
a_{ii}x_i^2
\end{equation*}
be the quadratic form associated with an $r\times r$ integral
positive definite symmetric matrix $(a_{ij})$ with even diagonal
entries. It is one of the important problems in number theory to
find the number of solutions of the equation
\begin{equation*}
Q(\mathbf{x})=n\quad(\mathbf{x}\in\mathbb{Z}^r)
\end{equation*}
for a given nonnegative integer $n$. In the case of $r=2$ it was
well studied by Fermat, Lagrange, Gauss and Dirichlet, and general
cases were considered systematically by Minkowski, Hasse and Siegel.
Only in very special cases does exist a satisfactory formula for the
representation number
\begin{equation*}
r_Q(n)=\#\{\mathbf{x}\in\mathbb{Z}^r~;~Q(\mathbf{x})=n\}.
\end{equation*}
For instance, when $Q=x_1^2+x_2^2+x_3^2+x_4^2+x_5^2+x_6^2$, Jacobi
(\cite[pp.159--170]{Jacobi}) gave the formula
\begin{equation*}
r_Q(n)=-4\sum_{d>0,d|n}\chi_{-4}(d)d^2+16\sum_{d>0,d|n}\chi_{-4}(n/d)d^2,
\end{equation*}
where $\chi_{-4}(d)=(\frac{-4}{d})$. This number $r_Q(n)$ is exactly
the Fourier coefficient of the corresponding Eisenstein series
(\cite[$\S$11.3]{Iwaniec}). One can also refer to
\cite[$\S$31]{Fine} for some concrete examples related to
hypergeometric series.
\par
Now, let us consider the quadratic form
\begin{equation*}
Q(x_1,x_2,x_3,x_4)=x_1^2+2x_2^2+x_3^2+x_4^2+x_1x_3+x_1x_4+x_2x_4\quad\textrm{associated
with the
matrix}~\left(\begin{smallmatrix}2&0&1&1\\0&4&0&1\\1&0&2&0\\1&1&0&2\end{smallmatrix}\right).
\end{equation*}
If $\Theta_Q(\tau)=\sum_{n=0}^\infty r_Q(n)q^n$ ($q=e^{2\pi i\tau}$)
is the theta function associated with $Q$, then $\Theta_Q(\tau)$
lies in the space $\mathcal{M}_{13}(13,\chi_{13})$ of dimension $2$
which consists of all modular forms for $\Gamma_0(13)$ associated
with the character $\chi_{13}(\cdot)=(\tfrac{13}{\cdot})$
(\cite[Corollary 4.9.5]{Miyake}). Eum et al (\cite[Example
3.4]{E-K-S}) recently provided a basis of the space
$\mathcal{M}_{2}(\Gamma_1(13))$ of dimension $13$, which consists of
modular forms for $\Gamma_1(13)$, in terms of Klein forms and
expressed $\Theta_Q(\tau)$ as a linear combination of the basis
elements. On the other hand, they happened to find in the process an
interesting identity
\begin{equation}\label{interesting}
r_Q(p^2n)=\frac{r_Q(p^2)r_Q(n)}{r_Q(1)}\quad\textrm{for any prime
$p$ other than $13$ and any positive integer $n$ prime to $p$}.
\end{equation}
But, they could give only a conditional proof by applying Hecke
operators on $\Theta_Q(\tau)$ as follows: if $p$ is a prime
satisfying the relation
\begin{equation}\label{conditional}
r_Q(p^2)=r_Q(1)(1+\chi_{13}(p)p+p^2),
\end{equation}
then (\ref{interesting}) is true (\cite[Proposition 4.3]{E-K-S}).
For example, each prime $p$ ($\neq13$) less than or equal to $347$
satisfies (\ref{conditional}). And so, the proof of
(\ref{interesting}) has remained open.
\par
In this paper, we shall completely prove the conjecture
(\ref{interesting}) (Theorem \ref{main}) by using the fact that the
space $\mathcal{M}_2(13,\chi_{13})$ is generated by two Eisenstein
series (Corollary \ref{two}). Also, we shall obtain a general
formula for $r_Q(n)$ that looks like Jacobi's formula, from which we
get a concise formula for $r_Q(n^2)$ for any integer $n$ (Remark
\ref{formula} and Table \ref{table}).

\section{Modular forms and Hecke operators}

Let $k$ be an integer. For each
$\gamma=\begin{pmatrix}a&b\\c&d\end{pmatrix}\in\mathrm{SL}_2(\mathbb{Z})$
we define the \textit{weight $k$ slash operator} $\cdot|[\gamma]_k$
on a function $f(\tau)$ on $\mathbb{H}$ (= the complex upper
half-plane) by
\begin{equation*}
f(\tau)|[\alpha]_k:=(c\tau+d)^{-k}(f(\tau)\circ\gamma),
\end{equation*}
where $\gamma$ acts on $\mathbb{H}$ as a fractional linear
transformation $\tau\mapsto(a\tau+b)/(c\tau+d)$. Let $\Gamma$ be one
of the following congruence subgroups
\begin{eqnarray*}
\Gamma_1(N)&:=&\bigg\{\alpha\in\mathrm{SL}_2(\mathbb{Z})~;~
\alpha\equiv\left(\begin{matrix}1&*\\0&1\end{matrix}\right)\pmod{N}\bigg\},\\
\Gamma_0(N)&:=&\bigg\{\alpha\in\mathrm{SL}_2(\mathbb{Z})~;~
\alpha\equiv\left(\begin{matrix}*&*\\0&*\end{matrix}\right)\pmod{N}\bigg\}
\end{eqnarray*}
for a positive integer $N$. A holomorphic function $f(\tau)$ on
$\mathbb{H}$ is called a \textit{modular form for $\Gamma$ of weight
$k$} if
\begin{itemize}
\item[(i)] $f(\tau)|[\gamma]_k=f(\tau)$ for all $\gamma\in\Gamma$,
\item[(ii)] $f(\tau)$ is holomorphic at every cusp ($\in\mathbb{Q}\cup\{\infty\}$)
(\cite[pp.125--126]{Koblitz}). In particular, since
$\left(\begin{smallmatrix}1&1\\0&1\end{smallmatrix}\right)\in\Gamma$
and
$f(\tau)\circ\left(\begin{smallmatrix}1&1\\0&1\end{smallmatrix}\right)=f(\tau+1)$
by (i), $f(\tau)$ has a Laurent series expansion with respect to
\begin{equation*}
q:=e^{2\pi i\tau}
\end{equation*}
of the form
\begin{equation*}
f(\tau)=\sum_{n=0}^\infty a(n)q^n\quad(a(n)\in\mathbb{C}),
\end{equation*}
which is called the \textit{Fourier expansion} of $f(\tau)$ (at the
cusp $\infty$).
\end{itemize}
Moreover, if a modular form vanishes at every cusp, it is called a
\textit{cusp form}. We denote the space of all modular forms
(respectively, cusp forms) for $\Gamma$ of weight $k$ by
$\mathcal{M}_k(\Gamma)$ (respectively, $\mathcal{S}_k(\Gamma)$).
\par
For a given Dirichlet character $\chi$ modulo $N$
 we define a character of $\Gamma_0(N)$ (\cite[pp.79--80]{Miyake}),
also denoted by $\chi$, to be
\begin{equation*}
\chi(\gamma):=\chi(d)\quad\textrm{for}~\gamma=\begin{pmatrix}a&b\\c&d\end{pmatrix}\in\Gamma_0(N).
\end{equation*}
Let
\begin{eqnarray*}
\mathcal{M}_k(N,\chi)&:=&\{f(\tau)\in \mathcal{M}_k(\Gamma_1(N))~;~
f(\tau)|[\gamma]_k=\chi(\gamma)f(\tau)\quad\textrm{for
all}~\gamma\in\Gamma_0(N)\},\\
\mathcal{S}_k(N,\chi)&:=&\mathcal{S}_k(\Gamma_1(N))\cap
\mathcal{M}_k(N,\chi),
\end{eqnarray*}
which are subspaces of $\mathcal{M}_k(\Gamma_1(N))$ and $\mathcal{S}_k(\Gamma_1(N))$,
respectively.
Then we have the decomposition
\begin{equation*}
\mathcal{M}_k(\Gamma_1(N))=\bigoplus_\chi \mathcal{M}_k(N,\chi),
\end{equation*}
where $\chi$ runs over all Dirichlet characters modulo $N$
\cite[Chapter III Proposition 28]{Koblitz}. If $\chi(-1)\neq(-1)^k$,
then the space $\mathcal{M}_k(N,\chi)$ is known to be $\{0\}$
(\cite[p.138]{Koblitz}).

\begin{proposition}\label{weight0}
Let $N$ be a positive integer.
\begin{itemize}
\item[(i)] $\dim_\mathbb{C}\mathcal{M}_k(\Gamma_1(N))=0$ for any negative integer $k$.
\item[(ii)] $\dim_\mathbb{C}\mathcal{M}_0(\Gamma_1(N))=1$, and hence
$\dim_\mathbb{C}\mathcal{M}_0(N,\chi)=0$ if $\chi$ is nontrivial.
\end{itemize}
\end{proposition}
\begin{proof}
See \cite[Theorems 2.5.2 and 2.5.3]{Miyake}.
\end{proof}

Using the Riemann-Roch Theorem, Cohen and Oesterl\'{e} (\cite{C-O})
explicitly computed the following dimension formulas.

\begin{proposition}\label{Cohen}
Let $k$ be an integer and $\chi$ be a Dirichlet character modulo $N$
for which $\chi(-1)=(-1)^k$. For each prime $p$ dividing $N$, let
$r_p$ \textup{(}respectively, $s_p$\textup{)} denote the power of
$p$ dividing $N$ \textup{(}respectively, the conductor of
$\chi$\textup{)}. Define
\begin{equation*}
\lambda(r_p,s_p,p):=\left\{
\begin{array}{ll}
p^{r'}+p^{r'-1} & \textrm{if}~2s_p\leq r_p=2r'\\
2p^{r'} & \textrm{if}~2s_p\leq r_p=2r'+1\\
2p^{r_p-s_p} & \textrm{if}~2s_p>r_p,
\end{array}\right.
\end{equation*}
and
\begin{equation*}
\nu_k:=\left\{\begin{array}{rl}0 & \textrm{if}~$k$~\textrm{is odd}\\
-1/4 & \textrm{if}~
k\equiv2\pmod{4}\\
1/4 & \textrm{if}~ k\equiv0\pmod{4},\end{array}\right.\quad
\mu_k:=\left\{\begin{array}{rl}0 & \textrm{if}~
k\equiv1\pmod{3}\\
-1/3 & \textrm{if}~
k\equiv2\pmod{3}\\
1/3 & \textrm{if}~k\equiv0\pmod{3}.\end{array}\right.
\end{equation*}
Then we have
\begin{equation*}
\dim_{\mathbb{C}}\mathcal{M}_{k}(N,\chi)-
\dim_{\mathbb{C}}\mathcal{S}_{2-k}(N,\chi)
=\frac{(k-1)N}{12}\prod_{p|N}(1+p^{-1})
+\frac{1}{2}\prod_{p|N}\lambda(r_p,s_p,p)-\nu_{2-k}\alpha(\chi)-\mu_{2-k}\beta(\chi),
\end{equation*}
where
\begin{equation*}
\alpha(\chi):=\sum_{\begin{smallmatrix}x\pmod{N}\\x^2+1\equiv0\pmod{N}\end{smallmatrix}}
\chi(x)\quad\textrm{and}\quad
\beta(\chi):=\sum_{\begin{smallmatrix}x\pmod{N}\\x^2+x+1\equiv0\pmod{N}\end{smallmatrix}}\chi(x).
\end{equation*}
\end{proposition}
\begin{proof}
See \cite[Th\'{e}or\`{e}m 1]{C-O} or \cite[Theorem 1.56]{Ono}.
\end{proof}

\begin{remark}\label{ab}
Suppose that $N$ is a prime. Since $r_N=1$ and $s_N=0$ or $1$, we
get $\lambda(r_N,s_N,N)=2$. Observe that there are at most two
$x\pmod{N}$ which satisfy $x^2+1\equiv0\pmod{N}$. Furthermore, since
$|\chi(x)|=1$, we deduce $|\alpha(\chi)|\leq2$. In a similar way, we
have $|\beta(\chi)|\leq2$. Hence
\begin{equation*}
\dim_{\mathbb{C}}\mathcal{M}_{k}(N,\chi)-
\dim_{\mathbb{C}}\mathcal{S}_{2-k}(N,\chi)
\geq\frac{(k-1)N}{12}\cdot(1+N^{-1})+\frac{1}{2}\cdot2-\frac{1}{4}\cdot2-\frac{1}{3}\cdot{2}
=\frac{(k-1)(N+1)}{12}-\frac{1}{6}.
\end{equation*}
\end{remark}

For a nonzero integer $N$ with $N\equiv0$ or $1\pmod{4}$ we denote
by $\chi_N$ the Dirichlet character modulo $|N|$ defined by
\begin{equation*}
\chi_N(d):=\textrm{the Kronecker
symbol}~\bigg(\frac{N}{d}\bigg)\quad\textrm{for}~d\in(\mathbb{Z}/|N|\mathbb{Z})^\times.
\end{equation*}
Note that
\begin{equation}\label{def-1}
\bigg(\frac{N}{-1}\bigg):=\left\{\begin{array}{rl}1 & \textrm{if}~N>0\\
-1 & \textrm{if}~N<0.\end{array}\right.
\end{equation}
In particular, let $N$ be the discriminant of a quadratic field,
namely, for a square-free integer $m$ ($\neq1$)
\begin{equation*}
N=\left\{\begin{array}{rl} m & \textrm{if}~m\equiv1\pmod{4}\\
4m & \textrm{if}~m\not\equiv1\pmod{4}.\end{array}\right.
\end{equation*}
Then $\chi_N$ becomes a primitive Dirichlet character modulo $|N|$
(\cite[pp.82--84]{Miyake}).

\begin{corollary}\label{dim2}
Let $k$ \textup{(}$\geq2$\textup{)} be an integer and $N$ be a prime
such that $(-1)^kN$ is the discriminant of a quadratic field. Then,
$\dim_\mathbb{C}\mathcal{M}_k(N,\chi_{(-1)^kN})=2$ if and only if
$(k,N)\in\{(2,5),(2,13),(2,17),(3,3),(4,5),(5,3)\}$.
\end{corollary}
\begin{proof}
Since $(-1)^kN$ is the discriminant of a quadratic field and $N$ is
a prime, $(-1)^kN\equiv1\pmod{4}$ and $\chi_{(-1)^kN}$ is a
primitive Dirichlet character modulo $N$. Suppose that
$\dim_\mathbb{C}\mathcal{M}_k(N,\chi_{(-1)^kN})=2$. We then see that
\begin{eqnarray*}
2&=&\dim_{\mathbb{C}}\mathcal{M}_{k}(N,\chi_{(-1)^kN})
-\dim_{\mathbb{C}}\mathcal{S}_{2-k}(N,\chi_{(-1)^kN}),
~\textrm{because
$\dim_\mathbb{C}\mathcal{S}_{2-k}(N,\chi_{(-1)^kN})=0$ by
Proposition
\ref{weight0},}\\
&\geq&\frac{(k-1)(N+1)}{12}-\frac{1}{6}\quad\textrm{by Remark
\ref{ab}}.
\end{eqnarray*}
It follows that $(k-1)(N+1)\leq26$, and so the possible pairs of
$(k,N)$ are
\begin{equation*}
(2,5),(2,13),(2,17),(3,3),(3,7),(3,11),(4,5),(5,3),(7,3).
\end{equation*}
Now, one can easily verify that
$\dim_\mathbb{C}\mathcal{M}_k(N,\chi_{(-1)^kN})=2$ except for
$(3,7),(3,11),(7,3)$ by Propositions \ref{weight0} and \ref{Cohen}.
\end{proof}

Let $\chi$ be a nontrivial primitive Dirichlet character modulo $N$.
The \textit{Dirichlet $L$-function} $L(s,\chi)$ on $s\in\mathbb{C}$
is defined by
\begin{equation*}
L(s,\chi):=\sum_{n=1}^\infty\frac{\chi(n)}{n^s},
\end{equation*}
where we set $\chi(n)=0$ if $\gcd(n,N)\neq1$. As is well-known, the
function converges for $\mathrm{Re}(s)>1$. Moreover, it extends to
an entire function and satisfies the following functional equation
\begin{equation*}
L(s,\chi)=L(1-s,\overline{\chi})\bigg(\frac{2\pi}{N}\bigg)^s\frac{S(\chi)}{\Gamma(s)}\bigg(\frac{e^{\pi
is/2}-\chi(-1)e^{-\pi is/2}}{e^{\pi is}-e^{-\pi is}}\bigg),
\end{equation*}
where
\begin{equation*}
S(\chi):=\sum_{a=1}^{N-1}\chi(a)e^{2\pi ia/N}\quad\textrm{and}\quad
\Gamma(s):=\int_0^\infty e^{-t}t^{s-1}dt
\end{equation*}
(\cite[Chapter XIV Theorem 2.2(ii)]{Lang}).

\begin{lemma}\label{L-function}
Let $k$ be a positive integer and $\chi$ be a nontrivial primitive
Dirichlet character modulo $N$.
\begin{itemize}
\item[(i)] $L(1-k,\chi)\neq0$ if and only if $\chi(-1)=(-1)^k$.
\item[(ii)] We have
\begin{equation*}
L(1-k,\chi)=-\frac{B_{k,\chi}}{k},
\end{equation*}
where $B_{k,\chi}$ is a \textit{generalized Bernoulli number}
defined by the following identity of infinite series
\begin{equation*}
\sum_{a=1}^{N-1}\chi(a)\frac{te^{at}}{e^{Nt}-1}=\sum_{k=0}^\infty
B_{k,\chi}\frac{t^k}{k!}.
\end{equation*}
\end{itemize}
\end{lemma}
\begin{proof}
(i) See \cite[Chapter XIV Corollary of Theorem 2.2]{Lang}.\\
(ii) See \cite[Chapter XIV Theorem 2.3]{Lang}.
\end{proof}

\begin{proposition}\label{Eisenstein}
Let $\chi$ and $\psi$ be primitive Dirichlet characters modulo $L$
and $M$, respectively. Let $k$ be an integer such that
$\chi(-1)\psi(-1)=(-1)^k$. Define
\begin{equation*}
E_{k,\chi,\psi}(\tau):=c_0+\sum_{n=1}^\infty\bigg(
\sum_{d>0,d|n}\chi(n/d)\psi(d)d^{k-1}\bigg)q^{n}\in\mathbb{C}[[q]]
\end{equation*}
with
\begin{equation*}
c_0:=\left\{\begin{array}{ll} 0 & \textrm{if}~L>1\\
L(1-k,\psi)/2 & \textrm{if}~L=1,
\end{array}\right.
\end{equation*}
and set $\chi(d)=\psi(d)=0$ if $\gcd(d,LM)\neq1$. Except for the
case when $k=2$ and $\chi=\psi=1$ the function
$E_{k,\chi,\psi}(\tau)$ defines an element of
$\mathcal{M}_k(LM,\chi\psi)$, which is called an \textit{Eisenstein
series}.
\end{proposition}
\begin{proof}
See \cite[Theorem 4.7.1 and Lemma 7.2.19]{Miyake}.
\end{proof}

\begin{corollary}\label{two}
Let $k$ and $N$ be positive integers such that $(-1)^kN$ is the
discriminant of a quadratic field. Then the Eisenstein series
\begin{eqnarray}
G_{k,N}(\tau)&:=&\frac{L(1-k,\chi_{(-1)^kN})}{2}+\sum_{n=1}^\infty\bigg(
\sum_{d>0,d|n}\chi_{(-1)^kN}(d)d^{k-1}\bigg)q^n=\frac{L(1-k,\chi_{(-1)^kN})}{2}+q+O(q^2),\phantom{aaaaa}
\label{G}\\
H_{k,N}(\tau)&:=&\sum_{n=1}^\infty\bigg(
\sum_{d>0,d|n}\chi_{(-1)^kN}(n/d)d^{k-1}\bigg)q^n=q+O(q^2)\label{H}
\end{eqnarray}
are linearly independent elements of
$\mathcal{M}_k(N,\chi_{(-1)^kN})$.
\end{corollary}
\begin{proof}
Set $\chi=1$ and $\psi=\chi_{(-1)^kN}$, which are primitive
Dirichlet characters modulo $1$ and $N$ ($\geq3$), respectively.
Since $\chi(-1)\psi(-1)=\chi_{(-1)^kN}(-1)=(-1)^k$ by the definition
(\ref{def-1}), it follows from Proposition \ref{Eisenstein} that the
Eisenstein series
\begin{equation*}
E_{k,\chi,\psi}(\tau)=
\frac{L(1-k,\chi_{(-1)^kN})}{2}+\sum_{n=1}^\infty\bigg(
\sum_{d>0,d|n}\chi_{(-1)^kN}(d)d^{k-1}\bigg)q^n=\frac{L(1-k,\chi_{(-1)^kN})}{2}+q+O(q^2)
\end{equation*}
belongs to $\mathcal{M}_k(N,\chi_{(-1)^kN})$. Observe that the
constant term $L(1-k,\chi_{(-1)^kN})/2$ does not vanish by Lemma
\ref{L-function}(i).
\par
Similarly, if we let $\chi=\chi_{(-1)^kN}$ and $\psi=1$, then the
Eisenstein series
\begin{equation*}
E_{k,\chi,\psi}(\tau)= \sum_{n=1}^\infty\bigg(
\sum_{d>0,d|n}\chi_{(-1)^kN}(d)d^{k-1}\bigg)q^n=q+O(q^2)
\end{equation*}
belongs to $\mathcal{M}_k(N,\chi_{(-1)^kN})$ by Proposition
\ref{Eisenstein}. This completes the proof.
\end{proof}

\begin{remark}
Although Corollary \ref{two} was originally given by Hecke
(\cite[p.818]{Hecke}), we derive it as a direct corollary of a more
generalized result (Proposition \ref{Eisenstein}) due to Miyake.
\end{remark}

Let $k$ be an integer and $\chi$ be a Dirichlet character modulo
$N$. For a positive integer $m$, the \textit{Hecke operator}
$\cdot|T_{m,k,\chi}$ is defined on the functions
$f(\tau)=\sum_{n=0}^\infty a(n)q^n\in\mathcal{M}_k(N,\chi)$ by the
rule
\begin{equation}\label{Heckeoperator}
f(\tau)|T_{m,k,\chi}:=\sum_{n=0}^\infty\bigg(\sum_{d>0,d|\gcd(m,n)}\chi(d)d^{k-1}a(mn/d^2)\bigg)q^n.
\end{equation}
Here we set $\chi(d)=0$ if $\gcd(N,d)\neq1$.

\begin{proposition}\label{Hecke}
With the notation as above, the operator $\cdot|T_{m,k,\chi}$
preserves the space $\mathcal{M}_k(N,\chi)$.
\end{proposition}
\begin{proof}
See \cite[Chapter 3 Propositions 36 and 39]{Koblitz}.
\end{proof}

\section{Theta functions associated with quadratic forms}

Let $A$ be an $r\times r$ positive definite symmetric matrix over
$\mathbb{Z}$ with even diagonal entries. Let $Q$ be its associated
quadratic form, namely
\begin{equation*}
Q(\mathbf{x})=\frac{1}{2}\mathbf{x}^TA\mathbf{x}\quad\textrm{for}~
\mathbf{x}=\begin{pmatrix}x_1\\\vdots\\x_r\end{pmatrix}\in\mathbb{Z}^r.
\end{equation*}
We define the theta function $\Theta_Q(\tau)$ on $\mathbb{H}$
associated with $Q$ by
\begin{equation*}
\Theta_Q(\tau):=\sum_{\mathbf{x}\in\mathbb{Z}^r}e^{2\pi
iQ(\mathbf{x})\tau} =\sum_{n=0}^\infty r_Q(n)q^n,
\end{equation*}
where
\begin{equation*}
r_Q(n):=\#\{\mathbf{x}\in\mathbb{Z}^r~;~Q(\mathbf{x})=n\}
\end{equation*}
is the representation number of $n$ by $Q$.

\begin{proposition}\label{theta}
With the notations as above, we further assume that $r$ is even. Let
$N$ be a positive integer such that $NA^{-1}$ is an integral matrix
with even diagonal entries. Then $\Theta_Q(\tau)$ belongs to
$\mathcal{M}_{r/2}(N,\chi_{(-1)^{r/2}\det(A)})$.
\end{proposition}
\begin{proof}
See \cite[Theorem 10.9]{Iwaniec} or \cite[Corollary 4.9.5]{Miyake}.
\end{proof}

\begin{remark}\label{detremark}
If such a matrix $A$ exists in the statement of Proposition
\ref{theta}, then $(-1)^{r/2}\det(A)\equiv0$ or $1\pmod{4}$
(\cite[p.180]{Iwaniec}). So $\chi_{(-1)^k\det(A)}$ makes sense.
\end{remark}

Now we are ready to prove our main theorem as follows.

\begin{theorem}\label{main}
Let $k$ and $N$ be positive integers such that $(-1)^kN$ is a
discriminant of a quadratic field and
$\dim_\mathbb{C}\mathcal{M}_k(N,\chi_{(-1)^kN})=2$. Let $A$ be a
$2k\times2k$ positive definite symmetric matrix over $\mathbb{Z}$
with $\det(A)=N$ such that both $A$ and $NA^{-1}$ have even diagonal
entries. Let $Q$ be the quadratic form associated with $A$.
\begin{itemize}
\item[(i)] We have $r_Q(0)=1$ and
\begin{equation}\label{generalformula}
r_Q(n)=c_1\sum_{d>0,d|n}\chi_{(-1)^kN}(d)d^{k-1}+c_2\sum_{d>0,d|n}\chi_{(-1)^kN}(n/d)d^{k-1}\quad\textrm{for
any positive integer $n$},
\end{equation}
where
\begin{equation}\label{c1c2}
c_1=\frac{2}{L(1-k,\chi_{(-1)^kN})}\quad\textrm{and}\quad
c_2=r_Q(1)-\frac{2}{L(1-k,\chi_{(-1)^kN})}.
\end{equation}
\item[(ii)] Let $p$ be a prime not dividing $N$. If $m$ is a nonnegative integer
such that $\chi_{(-1)^kN}(p^m)=1$ \textup{(}this condition holds
true whenever $m$ is even\textup{)}, then we have the identity
\begin{equation*}
r_Q(1)r_Q(p^{m}n)=r_Q(p^{m})r_Q(n)\quad\textrm{for any positive
integer $n$ prime to $p$}.
\end{equation*}
\end{itemize}
\end{theorem}
\begin{proof}
(i) Since $A$ is positive definite, $r_Q(0)=1$. Consider the theta
function $\Theta_Q(\tau)=\sum_{n=0}^\infty r_Q(n)q^n$. Then it lies
in $\mathcal{M}_{k}(N,\chi_{(-1)^kN})$ by Proposition \ref{theta}.
Since we are assuming that
$\dim_\mathbb{C}\mathcal{M}_{k}(N,\chi_{(-1)^kN})=2$, we see from
Corollary \ref{two} that
\begin{equation}\label{combination}
\Theta_Q(\tau)=c_1G_{k,N}(\tau)+c_2H_{k,N}(\tau) \quad\textrm{for
some}~c_1,c_2\in\mathbb{C}.
\end{equation}
We can then determine $c_1$ and $c_2$ in (\ref{c1c2}) by observing
the first two terms of
\begin{equation*}
\Theta_Q(\tau)=1+r_Q(1)q+O(q^2),\quad
G_{k,N}(\tau)=\frac{L(1-k,\chi_{(-1)^kN})}{2}+q+O(q^2)\quad\textrm{and}\quad
H_{k,N}(\tau)=q+Q(q^2).
\end{equation*}
By (\ref{combination}) and the definitions (\ref{G}), (\ref{H}) we
obtain a general formula (\ref{generalformula}) for $r_Q(n)$
($n\geq1$).\\
(ii) Let $p$ be a prime not dividing $N$ and $m$ be a nonnegative
integer such that $\chi_{(-1)^kN}(p^m)=1$. By the formula
(\ref{generalformula}) we get that
\begin{eqnarray}
r_Q(p^{m})&=&c_1\sum_{a=0}^{m}\chi_{(-1)^kN}(p^a)p^{a(k-1)}+c_2
\sum_{a=0}^{m}\chi_{(-1)^kN}(p^{m-a})p^{a(k-1)}\nonumber\\&=&
(c_1+c_2)\sum_{a=0}^{m}\chi_{(-1)^kN}(p^a)p^{a(k-1)}\quad\textrm{by
the facts $\chi_{(-1)^kN}(p^m)=1$ and
$\chi_{(-1)^kN}(p)=\pm1$}\nonumber\\
&=&r_Q(1)\sum_{a=0}^{m}\chi_{(-1)^kN}(p^a)p^{a(k-1)} \quad\textrm{by
(\ref{c1c2})}.\label{primecondition}
\end{eqnarray}
On the other hand, we deduce that
\begin{eqnarray*}
r_Q(1)\Theta_Q(\tau)|T_{p^{m},k,\chi_{(-1)^kN}}&=& r_Q(1)\bigg(
r_Q(0)\sum_{a=0}^{m}\chi_{(-1)^kN}(p^a)p^{a(k-1)}+r_Q(p^{m})q+O(q^2)\bigg)\quad
\textrm{by the definition (\ref{Heckeoperator})}\\
&=&r_Q(0)r_Q(p^{m})+r_Q(1)r_Q(p^{m})q+O(q^2)\quad\textrm{by (\ref{primecondition})}\\
&=&r_Q(p^{m})(r_Q(0)+r_Q(1)q+O(q^2)),
\end{eqnarray*}
which turns out to be an element of
$\mathcal{M}_k(N,\chi_{(-1)^kN})$ by Proposition \ref{Hecke}. Taking
the set $\{\Theta_Q(\tau)=1+r_Q(1)q+O(q^2)$,
$H_{k,N}(\tau)=q+O(q^2)\}$ as a basis of the space
$\mathcal{M}_k(N,\chi_{(-1)^kN})$ one can derive
\begin{equation*}
r_Q(1)\Theta_Q(\tau)|T_{p^{m},k,\chi_{(-1)^kN}}=r_Q(p^{m})\Theta_Q(\tau).
\end{equation*}
Therefore, comparing the Fourier coefficients of the term $q^n$ for
any positive integer $n$ prime to $p$ we achieve the identity
\begin{equation*}
r_Q(1)r_Q(p^{m}n)= r_Q(p^{m})r_Q(n)
\end{equation*}
as desired.
\end{proof}

\begin{remark}\label{formula}
\begin{itemize}
\item[(i)]
Let $p$ be a prime and $m$ be a nonnegative integer. If $p$ divides
$N$, then $\chi_{(-1)^kN}(p)=0$ and we get from
(\ref{generalformula}) that
\begin{equation}\label{N}
r_Q(p^m)=c_1+c_2p^{m(k-1)}.
\end{equation}
If $p$ does not divide $N$, then we see from (\ref{generalformula})
that
\begin{eqnarray}
r_Q(p^m)&=&c_1\sum_{a=0}^m\chi_{(-1)^kN}(p^a)p^{a(k-1)}+c_2\sum_{a=0}^m
\chi_{(-1)^kN}(p^{m-a})p^{a(k-1)}\nonumber\\
&=&(c_1+c_2\chi_{(-1)^kN}(p)^m)\sum_{a=0}^m(\chi_{(-1)^kN}(p)p^{k-1})^a
\quad\textrm{by the fact $\chi_{(-1)^kN}(p)^2=1$}\nonumber\\
&=&(c_1+c_2\chi_{(-1)^kN}(p)^m)\frac{1-(\chi_{(-1)^kN}(p)p^{k-1})^{m+1}}{1-\chi_{(-1)^kN}(p)p^{k-1}}.\label{notdiv}
\end{eqnarray}
\item[(ii)] Assume that $N$ is a prime.
Let $n$ ($\geq2$) be an integer with prime factorization
\begin{equation*}
n=N^m\prod_{i=1}^t p_i^{m_i}\quad(m,m_i\geq0).
\end{equation*}
If $r_Q(1)\neq0$, then we have by Theorem \ref{main}(ii)
\begin{equation*}
r_Q(n^2)=r_Q(N^{2m}\prod_{i=1}^t p_i^{2m_i})=
\frac{r_Q(N^{2m}\prod_{i=1}^{t-1}p_i^{2m_i})r_Q(p_t^{2m_t})}{r_Q(1)}=\cdots=
\frac{r_Q(N^{2m})\prod_{i=1}^t r_Q(p_i^{2m_i})}{r_Q(1)^{t}}.
\end{equation*}
Therefore by (\ref{N}) and (\ref{notdiv}) one can get a concise
formula for $r_Q(n^2)$.
\end{itemize}
\end{remark}

Let $(k,N)\in\{(2,5),(2,13),(2,17),(3,3),(4,5),(5,3)\}$ as in
Corollary \ref{dim2}. Then for each pair $(k,N)$ one can find some
matrices to which Theorem \ref{main} and Remark \ref{formula} can be
applied. However, it doesn't seem to be known how to find such
matrices systematically. We close this section by giving a table for
these examples.

\savebox\hvOBox{%
\tiny\begin{tabular}
{c|cccccc}\hline\\
$(k,N)$ & $(2,5)$ & $(2,13)$ & $(2,17)$ & $(3,3)$ & $(4,5)$ & $(5,3)$\\\\\\
$\begin{array}{l}2k\times 2k\\
\textrm{symmetric}\\\textrm{matrix}~A\\
\textrm{with}\\
\det(A)=N
\end{array}$ &
$\begin{pmatrix}2&1&1&1\\1&2&1&1\\1&1&2&1\\1&1&1&2\end{pmatrix}$ &
$\begin{pmatrix}2&0&1&1\\0&4&0&1\\1&0&2&0\\1&1&0&2\end{pmatrix}$ &
$\begin{pmatrix}2&1&0&0\\1&2&0&1\\0&0&2&1\\0&1&1&4\end{pmatrix}$ &
$\begin{pmatrix}
2&0&0&0&0&1\\
0&2&0&0&1&0\\
0&0&2&1&0&0\\
0&0&1&2&0&1\\
0&1&0&0&2&1\\
1&0&0&1&1&2
\end{pmatrix}$
& $\begin{pmatrix}
2&1&0&0&0&0&0&0\\
1&2&1&0&0&0&0&0
\\0&1&2&1&0&0&0&0\\
0&0&1&2&2&0&0&0\\
0&0&0&2&4&1&0&0\\
0&0&0&0&1&2&1&0\\
0&0&0&0&0&1&2&1\\
0&0&0&0&0&0&1&4\end{pmatrix}$ &
$\begin{pmatrix}
2&1&0&0&0&0&0&0&0&0\\
1&2&0&0&0&0&0&0&0&0\\
0&0&2&1&0&0&0&0&0&0\\
0&0&1&2&1&0&0&0&0&0\\
0&0&0&1&2&1&0&0&0&0\\
0&0&0&0&1&2&2&0&0&0\\
0&0&0&0&0&2&4&1&0&0\\
0&0&0&0&0&0&1&2&1&0\\
0&0&0&0&0&0&0&1&2&1\\
0&0&0&0&0&0&0&0&1&2
\end{pmatrix}$
\\\\\\
eigenvalues of $A$ & $1,1,1,5$ &
$\begin{array}{l}\frac{5}{2}\pm\frac{\sqrt{9+4\sqrt{3}}}{2},\phantom{\bigg|}\\
\frac{5}{2}\pm\frac{\sqrt{9-4\sqrt{3}}}{2}
\end{array}$ &
$\begin{array}{l}
\frac{5+\sqrt{2}}{2}\pm\frac{\sqrt{7+2\sqrt{2}}}{2},\phantom{\bigg|}\\
\frac{5-\sqrt{2}}{2}\pm\frac{\sqrt{7-2\sqrt{2}}}{2}
\end{array}$ &
$1,3,2\pm\frac{\sqrt{6}}{2}\pm\frac{\sqrt{2}}{2}$ &
$\begin{array}{l} \textrm{eight (distinct) positive zeros
of}\\x^8-20x^7+162x^6 -684x^5\\+1611x^4-2092x^3+1370x^2\\-352x+5
\end{array}$
& $\begin{array}{l}1,3,\\
\textrm{eight (distinct) positive zeros of}\\
x^8-18x^7+130x^6-486x^5+1007x^4\\
-1142x^3+646x^2-140x+1
\end{array}$
\\\\\\
$\begin{array}{l} \textrm{diagonal entries}\\
\textrm{of $NA^{-1}$}
\end{array}$ &
$4,4,4,4$ & $14,4,10,12$ & $12,14,10,6$ & $6,4,4,10,10,18$ &
$12,38,78,132,50,28,12,2$ & $2,2,12,42,90,156,60,36,18,6$
\\\\\\
$\begin{array}{l}\textrm{quadratic form $Q$}\\
\textrm{associated with $A$}\end{array}$ & $\begin{array}{l} x_1^2+x_2^2+x_3^2+x_4^2\\
+x_1x_2+x_1x_3+x_1x_4\\
+x_2x_3+x_2x_4+x_3x_4
\end{array}$
& $\begin{array}{l}x_1^2+2x_2^2+x_3^2+x_4^2\\
+x_1x_3+x_1x_4+x_2x_4
\end{array}$
&
$\begin{array}{l}x_1^2+x_2^2+x_3^2+2x_4^2\\
+x_1x_2+x_2x_4+x_3x_4
\end{array}$ &
$\begin{array}{l}x_1^2+x_2^2+x_3^2+x_4^2\\+x_5^2
+x_6^2+x_1x_6+x_2x_5\\
+x_3x_4+x_4x_6+x_5x_6
\end{array}$ &
$\begin{array}{l} x_1^2+x_2^2+x_3^2+x_4^2\\
+2x_5^2+x_6^2+x_7^2+2x_8^2\\
+x_1x_2+x_2x_3+x_3x_4\\
+2x_4x_5+x_5x_6+x_6x_7+x_7x_8
\end{array}$
& $\begin{array}{l}
x_1^2+x_2^2+x_3^3+x_4^2+x_5^2+x_6^2+2x_7^2\\
+x_8^2+x_9^2+x_{10}^2+x_1x_2+x_3x_4+x_4x_5\\
+x_5x_6+2x_6x_7+x_7x_8+x_8x_9+x_9x_{10}
\end{array}$
\\\\\\
$r_Q(1)$ & $20$ & $12$ & $8$ & $72$ &
$126$ & $246$\\\\\\
$L(1-k,\chi_{(-1)^kN})$ & $-\frac{2}{5}$ & $-2$ & $-4$ &
$-\frac{2}{9}$ &
$2$ & $\frac{2}{3}$\\\\\\
$\Theta_Q(\tau)$ &
$\begin{array}{l}-5G_{2,5}(\tau)+25H_{2,5}(\tau)\\
=1+20q+30q^2+\cdots\end{array}$ &
$\begin{array}{l}-G_{2,13}(\tau)+13H_{2,13}(\tau)\\
=1+12q+14q^2+\cdots
\end{array}$&
$\begin{array}{l}-\frac{1}{2}G_{2,17}(\tau)+\frac{17}{2}H_{2,17}(\tau)\\
=1+8q+24q^2+18q^3+\cdots
\end{array}$&
$\begin{array}{l}
-9G_{3,3}(\tau)+81H_{3,3}(\tau)\\
=1+72q+270q^2+\cdots
\end{array}$ &
$\begin{array}{l}G_{4,5}(\tau)+125H_{4,5}(\tau)\\
=1+126q+868q^2+\cdots
\end{array}
$ & $\begin{array}{l} 3G_{5,3}(\tau)+243H_{5,3}(\tau)\\
=1+246 q+3600 q^2+\cdots
\end{array}$
\\\\\\\hline
\end{tabular}
}
\hvFloat[%
    floatPos=p,%
    capPos=t,%
    capWidth=1,%
    rotAngle=90,%
    objectPos=c,%
    useOBox=true%
]{table}{}{Theta functions associated with quadratic
forms\label{table}}{tab:des}

\newpage

\bibliographystyle{amsplain}

\end{document}